\documentclass[12pt]{amsart}
\usepackage{amsfonts}
\usepackage{amssymb}
\usepackage{amscd}

\input xy
\xyoption{all}\usepackage{color}
\title[$0$-Schur algebras]{Degenerate $0$-Schur algebras and nil-Temperley-Lieb algebras}
\author{Bernt Tore Jensen, Xiuping Su and Guiyu Yang}

\thanks{This work was supported by EPSRC 1st grant EP/1022317/1 and NSFC 11671234. The first and the third authors
would like to thank the Algebra and Geometry group at the University of Bath for the hospitality during their visit to the department.
}
\newtheorem{theorem}{Theorem}[section]

\newtheorem{lemma}[theorem]{Lemma}
\newtheorem{proposition}[theorem]{Proposition}
\newtheorem{definition}[theorem]{Definition}
\newtheorem{example}[theorem]{Example}

\newtheorem{corollary}[theorem]{Corollary}
\newtheorem{remark}[theorem]{Remark}

\newcommand{\dvector}{{\underline{\mathrm{dim}}}}

\newcommand{\lra}{\longrightarrow}

\newcommand{\ra}{\rightarrow}
\newcommand{\sdp}{\times\kern-.2em\vrule height1.1ex depth-.0.5ex}
\newcommand{\epi}{\lra \kern-.8em\ra}

\newcommand{\zschur}{S_0(n, r)}

\newcommand{\lam}{\lambda}

\newcommand{\ro}{\mathrm{ro}}
\newcommand{\co}{\mathrm{co}}
\newcommand{\fkF}{\mathcal{ F}}

\newcommand{\GL}{\rm GL}

\newcommand{\bbn}{\mathbb{N}}

\newcommand{\diag}{\rm diag}

\newcommand{\bbz}{\mathbb{Z}}

\newcommand{\ut}{{\underline{t}}}

\normalbaselines
\begin{document}

\begin{abstract}
In \cite{JS} Jensen and Su constructed 0-Schur algebras on double flag varieties. The construction leads 
to a presentation of 0-Schur algebras using quivers with relations and 
the quiver approach naturally  gives rise  to a new class of algebras. That is, the
path algebras defined on the quivers of 0-Schur algebras with relations modified from 
the defining relations of 0-Schur algebras by a tuple of parameters $\ut$. In particular,
when all the entries of $\ut$ are 1, we have 0-Schur algerbas. When all the entries of
$\ut$ are zero, we obtain a class of degenerate 0-Schur algebras. We prove that 
the degenerate  algebras are associated graded algebras
and quotients of 0-Schur algebras. Moreover, we give a geometric interpretation of 
the degenerate algebras using double flag varieties, in the same spirit as \cite{JS}, 
and show how the centralizer algebras are 
related to nil-Hecke algebras and  nil-Temperly-Lieb algebras


\end{abstract}

\maketitle

\section{Introduction}

It is well known that the classical Schur algebras are specialisations
of $q$-Schur algebras (see \cite{DJ} and \cite{Donkin}) at $q=1$. Analogously, $0$-Schur algebras
are specialisations of $q$-Schur algebras at $q=0$.
The $0$-Schur algebras have been studied by Donkin \cite[\S2.2]{Donkin} in
terms of $0$-Hecke algebras of symmetric groups,  by
Krob and Thibon \cite{KT} in connection with noncommutative symmetric
functions, by Deng and Yang on their presentations and representation types  \cite{DY, DY2}.

A new  approach towards $0$-Schur algebras was investigated by Jensen and Su \cite{JS}, by considering the
monoid structure of the $0$-Schur algebras.
Inspired by Beilinson, Lusztig and MacPherson's geometric construction of $q$-Schur algebras \cite{BLM}
and Reineke's work on a monoid structure of Hall algebras \cite{Reineke},
Su defined a generic multiplication in the positive part of
 $0$-Schur algebras \cite{Su}. The generic multiplication  was then generalized by
Jensen and Su \cite{JS} to give a global geometric construction of  $0$-Schur algebras.
This geometric construction produces a monoid structure,
simplifies the multiplication and provides a new
approach to studying the structure of $0$-Schur algebras.
In \cite{JSY} we 
gave a  construction of  indecomposable
projective modules and studied homomorphism spaces
between projective modules  \cite{JSY}. In an ongoing work \cite{JSY3}, we study further irreducible maps,   
construct idempotents  and present  $0$-Hecke algebras and 
basic $0$-Schur algebras  using quivers with relations. 
Consequently, we obtain an alternative  account to the result on extension groups
between simple modules by  Duchamp, Hivert, and Thibon \cite{Thibon} (see also \cite{Fayers}).

The nature of  $0$-Schur algebras exposed in \cite{JS} leads to several interesting related algebras.
First of all, we can modify the generating relations of 0-Schur algebras, relying on multiple parameters $\ut$.
In particular, when all the parameters are 1, we recover 0-Schur algebras and when 
all parameters are 0, we  obtain a class of basic algebras. We prove that similar to 
 0-Schur algebras, these newly defined algebras have reduced paths as 
basis (see \cite{JSY}),  which are in one-to-one correspondence with $\GL(V)$-orbits in double flag varieties 
or certain sets of integral matrices. So they have the same dimension as the corresponding 0-Schur algebras and 
they are proved to be   degenerate 0-Schur algebras.
We further show that they are isomorphic to
  quotients of $0$-Schur algebras by naturally
defined ideals and to the associated graded algebras of 0-Schur algebras, which notably
have a natural structure as a filtered algebra. We also investigate the relation between their centralizer algebras  and
Nil-Temperly-Lieb algebras.


The remainder of this paper is organised as follows. In Section 2, we give a brief background on $q$-Schur and $0$-Schur algebras
and discuss how to view the $0$-Schur algebra as a filtered algebra. In Section 3, we prove some preliminary
results on a family of idempotent ideals. In Section 4, we construct a series of algebras $D_\ut(n,r)$ using quivers and modified relations of
$0$-Schur algebras and prove an isomorphism theorem between different $D_\ut(n,r)$. In Section 5, 
we first construct the associated graded algebras $DS_0(n, r)$ of 0-Schur algebras and give them a geometric interpretation. 
We then show that $D_{\ut}(n, r)$ for $\ut=\underline{0}$ is a degenerate 0-Schur algebras and prove our 
 main result that the three algebras, $DS_0(n, r)$, $D_{\underline{0}}(n,r)$ and the quotient of $S_0(n, n+r)$ modulo a natural 
idempotent ideal are isomorphic.  In Section 6, we discuss relations between centralizer algebras of the degenerate 
algebras nil-Hecke algebras and nil-Temperley-Lieb algebras.

\section{Background on the algebra $S_0(n, r)$}
\subsection{The algebra $S_q(n,r)$. }

We first recall the construction of quantised Schur algebras given by
Beilinson-Lusztig-MacPherson \cite{BLM}. Let $k$ be a field and
let $V$ be a $k$-vector space of dimension $r$. Let ${\mathcal{
F}}$ be the set of $n$-step flags
$$V_1\subseteq V_2\subseteq\cdots \subseteq V_n=V.$$
The natural action of $\GL(V)$ on
the vector space $V$  induces a diagonal action of $\GL(V)$ on $\fkF\times \fkF$ defined by $$g(f,
f')=(gf, gf'),$$ where $g\in \GL(V)$ and $f, f'\in \fkF$. Denote the orbit of $(f,f')$ by $[f,f']$.

 Let
$\Xi(n,r)$ be the set of $n\times n$-matrices $A=(a_{ij})_{i, j}$ with $a_{ij}$ nonnegative integers and $\sum_{1\leq i,j\leq
n}a_{ij}=r$. Then there is a bijection from $\fkF\times\fkF/\GL(V)$ to
$\Xi(n,r)$ sending the orbit of $(f,f')$ to $A=(a_{i,j})_{i, j}$ with
\begin{equation}\label{aij} a_{ij}=\dim_k\frac{V_i\cap
V_j'}{V_{i-1}\cap V_j'+V_i\cap V_{j-1}'}\;\;\text{for $1\leq i,j\leq
n$},\end{equation}
 where ${ f}=(V_1\subseteq V_2\subseteq\cdots \subseteq
V_n=V)$, ${ f}'=(V'_1\subseteq V'_2\subseteq\cdots \subseteq
V'_n=V)$ and $V_0=V_0'=0$ by convention.
Denote by $e_A$ the orbit in
$\fkF\times\fkF$ corresponding to $A$.
Consider the diagram
$$\xymatrix{\mathcal{F}\times \mathcal{F}\times \mathcal{F} \ar[d]^\pi \ar[r]^{\Delta\;\;\;\;\;\;\;\;\;\;} &
(\mathcal{F}\times\mathcal{F})\times (\mathcal{F}\times\mathcal{F}) \\ \mathcal{F}\times \mathcal{F}}$$
where $\pi(f,g,h)=(f,h)$ and $\Delta(f,g,h)=((f,g),(g,h))$.

Let $ \bbz[q]$ be the polynomial ring in $q$ over the ring of integers.
Following \cite[Prop.~1.2]{BLM}, for any given $A, B, C \in
\Xi(n,r)$, there is a polynomial $g_{A,B,C}\in \bbz[q]$
  such that for all
finite fields $k$, $$g_{A,B,C}(|k|) = \frac{\pi^{-1}(e_C)\cap
\Delta^{-1}(e_A\times e_B)}{|e_C|},$$
where $|k|$ and $|e_C|$ are  the cardinalities of the field $k$ and the orbit $e_C$ over $k$.

Following a remark by Du \cite{Duj}, the $q$-Schur algebra studied by Dipper and James in \cite{DJ} (see also \cite{Donkin}) can now be defined as follows.

\begin{definition} [{\cite{BLM}}] \label{def-q-Schur}
The quantised Schur algbra $S_{q}(n,r)$ is the free $\bbz[q]$-module with basis
$\{e_A \mid A\in \Xi(n,r)\}$ and with multiplication given by
$$e_A \cdot e_B=\sum_{C\in \Xi(n,r)}g_{A,B,C}(q)e_C, \;\;{\rm
for\; all}\; A,B \in \Xi(n,r).$$
\end{definition}

For an $n\times n$-matrix
$A=(a_{ij})$, define the row and column vectors of $A$ by
$$\mbox{ro}(A)=(\sum_{j=1}^n a_{1j}, \ldots,
\sum_{j=1}^n a_{nj}) \;\;\text{and}\;\;
\mbox{co}(A)=(\sum_{i=1}^n a_{i1}, \ldots, \sum_{i=1}^n
a_{in}).$$
Note that if  $g_{A,B,C}(q)\not=0$, then 
\begin{equation}\label{nonzero-condition}
\ro(A)=\ro(C),\,\co(A)=\ro(B)\;\;\text{and}\;\;\co(B)=\co(C).
\end{equation}

Let $\Lambda(n,r)$ be the set of compositions of $r$ into $n$ parts. For each
$\lambda=(\lambda_1,\ldots,\lambda_n)\in \Lambda(n,r)$, let $\diag(\lambda)$ denote the
diagonal matrix $\diag(\lambda_1,\ldots,\lambda_n)$ and write
$k_\lambda=e_{\diag(\lambda)}$. By definition, for each
$A\in\Xi(n,r)$,
\begin{equation}\label{idempotents}
k_\lambda\cdot e_A=\left\{\begin{array}{ll}
                      e_A,\;&\text{if $\lambda=\ro(A)$};\\
                     0,
                     &\text{otherwise}\end{array}\right.\;\;\text{and}\;\;
e_A\cdot k_\lambda=\left\{\begin{array}{ll}
                      e_A,\;&\text{if $\lambda=\co(A)$};\\
                     0, &\text{otherwise.}\end{array}\right.
\end{equation}
 Thus, $\sum_{\lambda\in \Lambda(n,r)}k_\lambda$ is the identity
 of $S_{q}(n,r)$.

Denote by $E_{ij}$ the elementary $n\times n$ matrix with a single nonzero entry $1$ at $(i, j)$.
We denote by $e_{i, \lam}$ (resp. $f_{j, \lam}$) the basis element of $S_q(n,r)$ corresponding to the matrix that has column vector $\lam$  and the only
nonzero off diagonal entry is $1$ at $(i, i+1)$ (resp. $(j+1, j)$).


\subsection{Definition of $S_0(n,r)$.}

From now on, let $k$ be algebraically closed.
In \cite{JS}, Jensen and Su defines a generic multiplication of orbits in $\mathcal{F}\times \mathcal{F}$
given by
\begin{equation}\label{mul}e_A\cdot e_B =\left \{\begin{tabular}{ll}  $e_C$ & if $\Delta^{-1}(e_A\times e_B)\neq \emptyset$,\\ $0$ & {otherwise},  \end{tabular}\right.\end{equation}
where $e_C$ is the unique open orbit in $\pi\Delta^{-1}(e_A\times e_B)$.
This defines an associative $\mathbb{Z}$-algebra $G(n, r)$ with basis $\Xi(n,r)$ and in fact $G(n, r)$ is 
isomorphic to the
$0$-Schur algebra (Theorem 7.2.1 in \cite{JS})
$$S_0(n,r)=S_{q}(n,r)\otimes_{\bbz[q]}\bbz,$$
where $\bbz$ is viewed as the  $\bbz[q]$-module $\bbz[q]/\langle q\rangle$. 
As the multiplication in $G(n, r)$ is much simplified (e.g. the multiplication of two orbits is either  $0$ or again 
an orbit), in the rest part of this paper, we will take $G(n, r)$ as the $0$-Schur algebra $S_0(n,r)$.

\subsection{The fundamental multiplication rules}

Note that $S_0(n,r)$ is generated by $e_{i, \lam}$, $f_{i, \lam}$ and $k_\lam$,
where $1\leq i\leq n-1$ and $\lam \in \Lambda(n, r)$ (see
Lemma 6.9 in \cite{JS}).
Let
$$
e_i=\sum_{\lambda \in \Lambda(n, r)}e_{i, \lambda} \;\mbox{ and }\; f_i=\sum_{\lambda \in \Lambda(n, r)}f_{i, \lambda}.
$$
Note that for any given orbit $e_A$, by the definition of the  multiplication, $$e_i e_A=e_{i, \mathrm{co}(A)}e_A.$$
This says that  only one term remains in the product and the same for $e_A e_i$, $e_Af_i$ and $f_ie_A$.
The following are the fundamental multiplication rules in $S_0(n,r)$,  which describe the
action of generators on basis elements.

\begin{lemma}[Lemma 6.11, \cite{JS}] \label{FundMult}
Let $e_A\in S_0(n,r)$ with $\ro(A)=\lambda$.
\begin{itemize}
\item[(1)] If $\lam_{i+1}>0$, then $e_{i} e_A=e_X$, where
$X=A+E_{i,p}-E_{i+1,p}$ with $p=\mathrm{max}\{j\mid a_{i+1,j}>0\}$.
\item[(2)] If $\lam_{i}>0$, then $f_{i} e_A=e_Y$, where
$Y=A-E_{i,p}+E_{i+1,p}$ with $p=\mathrm{min}\{j \mid a_{i,j}>0\}$.
\end{itemize}
\end{lemma}

Symmetrically, there are the following formulas.

\begin{lemma} [Lemma 2.2, \cite{JSY}] \label{DFundMult}
Let $e_A\in S_0(n,r)$ with $\co(A)=\mu$.
\begin{itemize}
\item[(1)] If $\mu_{i+1}>0$, then $ e_Af_i=e_X$, where
$X=A+E_{p, i}-E_{p, i+1}$ with $p=\mathrm{max}\{j\mid a_{j, i+1}>0\}$.
\item[(2)] If $\mu_{i}>0$, then $ e_Ae_i=e_Y$, where
$Y=A-E_{p, i}+E_{p, i+1}$ with $p=\mathrm{min}\{j\mid a_{j, i}>0\}$.
\end{itemize}
\end{lemma}

\subsection{Presenting $S_0(n,r)$ by quiver with relations.} \label{presentation}

Let $$\alpha_i=(0,\ldots,0,1,-1,0,\ldots,0)\in\mathbb{Z}^n,$$ where  the only nonzero
entries $1$ and $-1$  are at the $i$th- and $(i+1)$th-positions, respectively.
We define a quiver $\Sigma(n, r)$ with vertices corresponding to $\lambda\in
\Lambda(n,r)$ and arrows
$$\xymatrix{k_{\lambda+\alpha_i} \;
\ar@/^15pt/[rr]^{f_{i, \lambda+\alpha_i}} & \;  \;  \; \;
\;  \;  \;  \;  \; \;  \;  \;  & \ar@/^15pt/[ll]^{e_{i,\lambda} } \;
 k_{\lambda}}$$
and let $\mathcal{J}\subseteq \mathbb{Z}\Sigma(n,r)$  be the ideal generated by
the binomial relations

$$P_{ij, \lambda}=k_{\mu}P_{ij}k_{\lambda},$$
$$N_{ij, \lambda}=k_{\mu}N_{ij}k_{\lambda}, $$ $$C_{ij,\lambda}=
k_{\lambda+\alpha_i-
\alpha_j}
C_{ij}k_{\lambda},$$
where
$$P_{ij}=\left\{  \begin{matrix} e_i^2e_j-
e_ie_je_i & \mbox{ for } i=j-1, \\
-e_ie_je_i+ e_je^2_{i} & \mbox{ for } i=j+1, \\ e_{i}e_{j}-e_{j} e_{i} &
\mbox{ otherwise},\end{matrix}\right.
\;\text{ and }\;\mu=\left\{ \begin{matrix}\lambda+2\alpha_i+
\alpha_j & \text{if } i=j\pm 1\\
\lambda + \alpha_i+ \alpha_j & \text{otherwise}; 
\end{matrix}
\right.
$$
$$N_{ij}=\left\{  \begin{matrix}
-f_if_jf_i+ f_jf^2_{i}& \mbox{ for }  i=j-1, \\ f_i^2f_j-
f_if_jf_i &\mbox{ for } i=j+1, \\ f_{i}f_{j}-f_{j} f_{i}
&\mbox{ otherwise}, \end{matrix}\right. 
\;\text{ and }\;\mu=\left\{ \begin{matrix}\lambda-2\alpha_i-
\alpha_j & \text{if } i=j\pm 1\\
\lambda-\alpha_i- \alpha_j & \text{otherwise}; 
\end{matrix}
\right.
$$
and 
$$
C_{ij}=e_{i}f_{j}-f_{j}e_{i}-\delta_{ij}(\sum_{\lam_{i+1}=0} k_{\lam}-\sum_{\lam_{i}=0}k_\lam).
$$

The quotient algebra $\mathbb{Z}\Sigma(n,r) / \mathcal{J}$ is isomorphic to
$S_0(n, r)$, by an isomorphism that maps the arrows $e_{i,\lambda}, f_{i,\lambda}$
and vertices $k_\lambda$ to the corresponding elements in $S_0(n,r)$ (Theorem 7.1.2 in \cite{JS}).
The relations $P_{ij, \lambda}, N_{ij, \lambda}$ are usually called the Serre relations of $S_0(n,r)$.
The relations $C_{ij, \lambda}$  is a  commutative 
relation when  $i\not= j$ or $\lambda_i\lambda_{i+1}\not= 0$,  an idempotent relation when exactly 
one of $\lambda_i$, $\lambda_{i+1}$ is zero, 
and an empty relation otherwise.

\subsection{Bases of reduced paths and $S_0(n,r)$ as a filtered algebra.}\label{grading}

We say that a path in $\Sigma(n,r)$ is reduced
if it is not equal to a path of shorter length in $\Sigma(n,r)$  modulo $\mathcal{J}$.
Equivalence classes of reduced paths form a multiplicative basis for $S_0(n,r)$,
which coincides with the basis $e_A$, $A\in \Xi(n,r)$. There are in general many
paths equal to $e_A$, modulo $\mathcal{J}$, but by the
fundamental multiplication rules and the presentation of $S_0(n,r)$,
the numbers of occurrences of $e_i$ and
$f_i$ in any reduced path are
$$E(e_A)_i = \sum_{l\leq i<m} a_{l,m}
\mbox{ and } F(e_A)_i = \sum_{m\leq i< l}a_{l,m},$$
respectively.

\begin{lemma} \label{subadd} For any $1\leq i\leq n-1$, we have
$$E(e_A)_i+E(e_B)_i-E(e_A\cdot e_B)_i=F(e_A)_i+F(e_B)_i-F(e_A\cdot e_B)_i\geq 0.$$
\end{lemma}
\begin{proof}
The generating relations $P_{ij}$ and $N_{ij}$ do not change the length of a path. The
length of a path decreases, when we apply the relations $C_{ii}$, that is, we replace
an $e_if_i$ or $f_ie_i$ by an idempotent. Note that in this case, where the length of a path
does decrease, the numbers of the occurrences of $e_i$ and $f_i$ decrease at the same
pace. Therefore
$$E(e_A)_i+E(e_B)_i-E(e_A\cdot e_B)_i=F(e_A)_i+F(e_B)_i-F(e_A\cdot e_B)_i\geq 0,$$
as stated.
\end{proof}

We define vectors $E(e_A)$ and $F(e_A)$ by
$$E(e_A)=(E(e_A)_i)_i \;\mbox{ and } \; F(e_A)=(F(e_A)_i)_i.$$
If we compare tuples $(E(e_A),F(e_A))$ componentwise, we have the following.

\begin{corollary}\label{cor1}
$S_0(n,r)$ is a filtered algebra with degree function $e_A\mapsto (E(e_A),F(e_A))$.
\end{corollary}

For $j\geq i$, let
$$e(i,j)=e_i\cdot e_{i+1}\cdot ... \cdot e_j\mbox{ and }f(j,i)=f_j\cdot ... \cdot f_{i+1}\cdot f_i.$$
We give two explicit
descriptions of reduced paths equal to a basis element $e_A$, analogous to the
monomial and PBW-basis of $S_q(n,r)$, respectively.

\begin{lemma} \label{monomial} Let $e_A\in S_0(n,r)$ with $A=(a_{i,j})_{i,j}$. Then we can write $e_A$ as reduced paths as follows.
\begin{itemize}
\item[(1)] $e_{A}=
(\prod_{s=n-1}^{1}\prod_{l=1}^s e_l^{\sum_{1\leq  p \leq l}a_{p, s+1}})
\cdot (\prod_{s=1}^{n-1}\prod_{l=n-1}^{s}f_{l}^{\sum_{l<p\leq  n}a_{p,s}})
\cdot k_{\co(A)}$;
\item[(2)] $e_A = (\prod_{j=n}^{2} \prod_{i=j-1}^{1} e(i,j-1)^{a_{ij}})\cdot
(\prod_{j=1}^{n-1} \prod_{i=n}^{j+1} f(i-1,j)^{a_{ij}}) \cdot k_{\co(A)}$.
\end{itemize}
\end{lemma}
\begin{proof}
The formula follows by repeatedly applying the fundamental multiplication rules
in Lemma \ref{FundMult}. In both (1) and (2), the numbers of occurrences
of $e_i$ and $f_i$ are  $E(e_A)_i$ and $F(e_A)_i$. So the paths are reduced.
\end{proof}

We explain by an example on how to achieve the above formulae.

\begin{example}
Let $A=(a_{ij})$, where $$A=\left(
         \begin{array}{ccc}
           0 & 1 &2 \\
           3 & 0 & 4\\
           5 & 6 & 0 \\
         \end{array}
       \right) \mbox{ and } B=\left(\begin{array}{ccc}
8 & 0 & 0 \\ 0 & 7 & 0 \\ 0 & 0 & 6\end{array}\right).$$
Note that $e_B = k_{\co(A)}$.
Using the fundamental multiplication
rules, we have the following.
$$(1) \;\;\;\;\;\;B=\left(\begin{array}{ccc}
8 & 0 & 0 \\ 0 & 7 & 0 \\ 0 & 0 & 6\end{array}\right) \stackrel{f_2^6}\longrightarrow
\left(\begin{array}{ccc}
8 & 0 & 0 \\ 0 & 1 & 0 \\ 0 & 6 & 6\end{array}\right)\stackrel{f_1^8}\longrightarrow
\left(\begin{array}{ccc}
0 & 0 & 0 \\ 8 & 1 & 0 \\ 0 & 6 & 6\end{array}\right)\stackrel{f_2^5}\longrightarrow$$
$$\left(\begin{array}{ccc}
0 & 0 & 0 \\ 3 & 1& 0 \\ 5 & 6 & 6\end{array}\right)\stackrel{e_1}\longrightarrow
\left(\begin{array}{ccc}
0 & 1 & 0 \\ 3 & 0 & 0 \\ 5 & 6 & 6\end{array}\right)\stackrel{e_2^6}\longrightarrow
\left(\begin{array}{ccc}
0 & 1 & 0 \\ 3 & 0 & 6 \\ 5 & 6 & 0\end{array}\right)\stackrel{e_1^2}\longrightarrow\left(
         \begin{array}{ccc}
           0 & 1 &2 \\
           3 & 0 & 4\\
           5 & 6 & 0 \\
         \end{array}
       \right)=A.$$

That is, 

$$
\aligned e_A
=&\;e_1^2\cdot e_2^6\cdot e_1\cdot f_2^5\cdot f_1^8\cdot f_2^6 \cdot e_B\\
=&\;e_1^{a_{13}}\cdot e_2^{(a_{13}+a_{23})}\cdot e_1^{a_{12}}\cdot f_2^{a_{31}}\cdot f_1^{(a_{21}+a_{31})}\cdot f_2^{a_{32}} \cdot e_B.\\\\
\endaligned
$$

\vspace{2mm}

$$(2)  \;\;\;\;\;\;  B=\left(\begin{array}{ccc}
8 & 0 & 0 \\ 0 & 7 & 0 \\ 0 & 0 & 6\end{array}\right) \stackrel{f_2^6}\longrightarrow
\left(\begin{array}{ccc}
8 & 0 & 0 \\ 0 & 1 & 0 \\ 0 & 6 & 6\end{array}\right)\stackrel{(f_2f_1)^5}\longrightarrow
\left(\begin{array}{ccc}
3 & 0 & 0 \\ 0 & 1 & 0 \\ 5 & 6 & 6\end{array}\right)\stackrel{f_1^3}\longrightarrow$$
$$\left(\begin{array}{ccc}
0 & 0 & 0 \\ 3 & 1 & 0 \\ 5 & 6 & 6\end{array}\right)\stackrel{e_1}\longrightarrow
\left(\begin{array}{ccc}
0 & 1 & 0 \\ 3 & 0 & 0 \\ 5 & 6 & 6\end{array}\right)\stackrel{(e_1e_2)^2}\longrightarrow
\left(\begin{array}{ccc}
0 & 1 & 2 \\ 3 & 0 & 0 \\ 5 & 6& 4\end{array}\right)\stackrel{e_2^4}\longrightarrow\left(
         \begin{array}{ccc}
           0 & 1 &2 \\
           3 & 0 & 4\\
           5 & 6 & 0 \\
         \end{array}
       \right)=A.$$

That is,  

$$
\aligned e_A=& \;e_2^4\cdot (e_1e_2)^2\cdot e_1\cdot f_1^3\cdot (f_2f_1)^5\cdot f_2^6\cdot e_B\\
=& \;e(2,2)^4\cdot e(1,2)^2\cdot e(1,1)\cdot f(1,1)^3\cdot f(2,1)^5\cdot f(2,2)^6\cdot e_B\\
=& \;e(2,2)^{a_{23}}\cdot e(1,2)^{a_{13}}\cdot e(1,1)^{a_{12}}\cdot f(1,1)^{a_{21}}\cdot f(2,1)^{a_{31}}\cdot f(2,2)^{a_{32}}\cdot e_B.
\\
\endaligned
$$

\vspace{2mm}

In both processes, the the lower triangular parts/upper triangular parts are created column-wise. The main differences are, 
for instance in the lower triangular parts, in (2) each step creates an entry $a_{ij}$ at a time, starting from the lowest entry and then moving upwards, while in (1) it goes downwards and in each step we apply the maximal times of $f_i$ so that afterwards 
we have exactly right entry in row $i$.
\end{example}

\section{Idempotents and idempotent ideals of $S_0(n, r)$}\label{ideals}

The set of compositions $\Lambda(n,r)$ (i.e. the vertices in the quiver $\Sigma(n,r)$)
can be drawn on an $(n-1)$-simplex, where
compositions with zero entries lie on the boundary. We call an idempotent
$k_\lambda$ boundary if $\lambda$ lies on the boundary of the simplex, and interior otherwise.
In this section, we are interested in the ideal generated by all boundary
idempotents. We obtain a dimension formula
for the quotient algebra, which turns out to be useful when we consider
degenerate $0$-Schur algebras Section 5.

Let $I_i(n, r)$ be the ideal generated by the idempotents $k_\lam$ with
 the number of nonzero entries in $\lam$ less than or equal to $i$.
That is, $I_i(n,r)$ is generated by idempotents corresponding to
$(i-1)$-faces of the simplex.
There are in general
several $(i-1)$-faces in $\Sigma(n,r)$, but as the following lemma shows, any one of them contains enough idempotents to generate the whole of $I_i(n,r)$.

\begin{lemma} The ideal 
$I_i(n,r)$ is generated by all idempotents $k_\lam$ lying in any chosen
$(i-1)$-face in $\Sigma(n,r)$.
\end{lemma}
\begin{proof}
For any $\lambda=(\dots, \lambda_{i-1}, 0, \lambda_{i+1}, \dots)\in \Lambda(n,r)$,
$$k_\lambda=f_{i}^{\lambda_{i+1}}{k_\mu }e_i^{\lambda_{i+1}},$$
where $\mu=(\dots, \lambda_{i-1}, \lambda_{i+1}, 0,   \dots)$ with
the other entries  equal to those of $\lambda$.
Therefore $k_{\lambda}$ is contained in the ideal generated by $k_{\mu}$. Similarly, $k_{\mu}$ is  contained
in the ideal generated by $k_{\lambda}$, and so
these two ideals are equal. This shows that
$I_i(n,r)$ is generated by all
idempotents $k_\lam$ lying in any chosen
$(i-1)$-face in $\Sigma(n,r)$.
\end{proof}

Denote by $I(n, r)$  the  ideal of $\zschur$ generated by the  boundary idempotents. If $r\geq n$, then $I(n, r)=I_{n-1}(n, r)$. If $r<n$, then $I(n, r)=\zschur$.

\begin{lemma}\label{descriptionofideals}
The ideal $I(n, r)$ has a basis consisting of $e_{A}$, where $A\in \Xi(n,r)$
is a matrix with at least one diagonal entry equal to $0$.
\end{lemma}

\begin{proof}
Denote by $S$ the subspace of $\zschur$ spanned by  $e_{A}$ with  at least one diagonal entry of $A$ equal to  $0$. We first show that $S$ is an ideal. It is enough to prove
that for any generator $x$ and any $e_D\in  S$,
both $xe_D$ and $e_Dx$ are contained in $S$. We only prove that $xe_D\in S$ for  $x=e_j$ and $D=(d_{ij})_{ij}$ with $d_{ii}=0$,
as the other cases can be done similarly.
Let $e_X=e_je_D$ and $X=(x_{ij})_{ij}$. By Lemma \ref{FundMult},  multiplying $e_j$ with $e_D$ from the left  only affects two entries
 in the $(j+1)$th-row and $j$th-row
in $D$, respectively.  We have either $x_{ii}=d_{ii}=0$ or $x_{ii}=1$, which occurs only if
$d_{i+1, i+1}=0$, but then $$x_{i+1, i+1}=0$$ as well.
In either case $xe_D=e_X\in S$, as required, and so $S$ is an ideal in $S_0(n,r)$.

As $I(n,r)$ is generated by idempotents $k_\lambda$ with $\lambda_i=0$ for some $i$,
we have $k_\lambda\in S$ and therefore $I(n,r)\subseteq S$, since $S$ is an ideal.

It remains to show that $S\subseteq I(n, r)$. Suppose that  $e_A\in S$ has the
diagonal entry $a_{ii}=0$, for some $1\leq i\leq n$. We will show that $e_A\in I(n,r)$.
Let $B$ be the matrix with

$$
b_{ss}=\left\{  \begin{tabular}{ll}
$\sum_{s\leq l\leq i}a_{ls}$ & if $s<i$,\\
0& if $s=i$,\\
$\sum_{i\leq l\leq s}a_{ls}$ & if $s>i$.\\
\end{tabular}\right.
$$
on the diagonal and

$$b_{st}=\left\{  \begin{tabular}{ll}
$0$ & if $i\geq s>t$ \\
$0$ & if $i\leq s < t$ \\
$a_{st}$ & otherwise
\end{tabular}\right.
$$
off the diagonal. The $i$'th row in $B$ is zero, and so we have $e_B=k_{\lam}e_B \in I(n, r)$, where $\lam=\ro(B)$. To complete the proof we will construct elements $x$ and $y$ such that $xye_B=e_A$, which implies that $e_A\in I(n,r)$. The element $y$ is a product of generators $e_j$ for $j>i$ 
and $x$ is a product of $f_j$ for $j\leq i$.
Multiplying $e_B$ with $x$ and $y$ produces the right entries at $(s, t)$ in the two zero region $i\geq s >t$ and $i\leq s<t$ from
the diagonal entries in $B$.
Explicitly, similar to Lemma \ref{monomial} we have $$
y=\prod_{s=n-1}^{i}\prod_{l=i}^s e_l^{\sum_{i\leq  p \leq l}a_{p, s+1}}
=\underbrace{(e_i^{a_{in}}e_{i+1}^{a_{in}+a_{i+1, n}}\dots e_{n-1}^{\sum_{i\leq p\leq n-1}a_{pn}})}_{s=n-1} \dots \underbrace{e_i^{a_{i, i+1}}}_{s=i},
$$
and similarly for $x$. Now by  the fundamental multiplication rules in Lemma \ref{FundMult} $$xye_B=e_A,$$ as required.
\end{proof}

The lemma shows that $I(n,r)$ is a summand of $S_0(n,r)$ as $\mathbb{Z}$-modules,
and so $S_0(n,r)/I(n,r)$ is a free $\mathbb{Z}$-module. We have the following formulae
which will be used Section 5.

\begin{corollary}\label{dim1}
\begin{itemize}\item[]
\item[(1)]
$\mathrm{rank\;} I(n, r)= \sum_{s=1}^n  \left( \begin{matrix}
n\\ s
\end{matrix}\right)  \left( \begin{matrix}
n^2+r-n-1\\ r+s-n
\end{matrix}\right). $
\item[(2)] $\mathrm{rank\;} \zschur/I(n, r)= \left( \begin{matrix}
n^2+r-n-1\\ r-n
\end{matrix}\right) .$
\end{itemize}
\end{corollary}
Note that in part (1), each term in the sum counts the matrices with exactly $s$ zero diagonal entries.

\section{ Modified algebras of $S_0(n,r)$}\label{degalg}

In the remaining of this paper, we will work with algebras defined over
a field $\mathbb{F}$. By modifying generating relations in Section \ref{presentation}, we have 
a series of modified algebras $D_\ut(n, r)$
of $S_0(n,r)$.
We will show that for a particular $\ut$, $D_\ut(n, r)$ is a degeneration of $S_0(n, r)$ in next section.

Let $B(n,r)\subseteq \Lambda(n,r)$ be the set of elements corresponding to boundary idempotents.
i.e., this is the set of $\lam$ such that there is some $\lam_i=0$.
Let $$\ut=(t_{i, \lam})_{ 1\leq i\leq n-1, \lam\in B(n,r) }$$ with each entry of $\ut$ in  $\mathbb{F}$. Denote by
$\mathcal{J}(\ut)$  the ideal of $\mathbb{F}\Sigma(n,r)$ generated by $P_{ij, \lambda}$, $N_{ij, \lambda}$ and
$C_{ij,\lambda}(\ut),$ where $P_{ij, \lam}$ and $N_{ij, \lam}$ are defined as in Section \ref{presentation} and
\begin{equation}\label{def idem}C_{ij,\lambda}(\ut) = k_{\lambda+\alpha_{i}-\alpha_j}
C_{ij}(\ut)k_{\lambda},\end{equation}
where $$
C_{ij}(\ut)=e_{i}f_{j}-f_{j}e_{i}-\delta_{ij}\cdot t_{i, \lam}\cdot(\sum_{\lam_{i+1}=0} k_{\lam}-\sum_{\lam_{i}=0}k_\lam).
$$
Note that when $\lambda_{i}=\lambda_{i+1}=0$, then $C_{ij, \lambda}(\ut)$ is an empty relation.

\begin{definition}
 $D_\ut(n,r):=\mathbb{F}\Sigma(n,r)/\mathcal{J}(\ut)$.
 \end{definition}

\begin{remark}
It is enough to define $\ut$ only on vertices corresponding to $B(n,r)$.
Note that for $i\neq j$, $C_{ij}(\ut)=e_if_j-f_je_i$. By (\ref{def idem}) there is a commutative diagram as follows.
$$\xymatrix{
  k_\lam \ar[d]_{f_j} \ar[r]^{e_i}
                & k_{\lambda+\alpha_{i}}  \ar[d]^{f_j}  \\
  k_{\lambda-\alpha_j}  \ar[r]_{e_i}
                & k_{\lambda+\alpha_{i}-\alpha_j}          }$$
When $i=j$ and either $\lambda_i\not=0$ or $\lambda_{i+1}\not=0$, we have the following three cases.

\begin{equation}\label{cii}C_{ii}(\ut)=\left\{
                                                            \begin{array}{ll}
                                                              e_if_i-f_ie_i, & \hbox{if \;$\lam_i\lam_{i+1}\neq 0$;} \\
                                                              e_if_i-t_{i, \lam}k_\lam, & \hbox{if \;$\lam_{i+1}=0$;} \\
                                                              t_{i, \lam}k_\lam-f_ie_i, & \hbox{if \;$\lam_i=0$.}
                                                            \end{array}
                                                          \right.
\end{equation}

For the first case, we have the commutative relation at $k_\lam$ as follows.

$$\xymatrix@!=0.02cm @M=0pt{
 &&&k_{\lam-\alpha_i}\ar@/^0.2pc/[rr]^{e_i}&&\ar@/^0.2pc/[ll]^{f_i}\ar@/^0.2pc/[rr]k_\lam
 \ar@/^0.2pc/[rr]^{e_i}&&\ar@/^0.2pc/[ll]^{f_i}k_{\lam+\alpha_i}
 \\
 &&&&&&&&\\
  }$$

Note that for the second and the third case, $\lam$ is boundary and the relations are not admissible if $t_{i, \lam}\neq 0$.
When all the $t_{i, \lam}=1$, we recover  0-Schur algebra $S_0(n,r)$. When all the $t_{i, \lam}=0$, we have a basic algebra and we will
prove in the following section that this is a degenerate 0-Schur algebra.
 \end{remark}

Note that "bad" parameter $\underline{t}$ may produce "bad" relations, in the sense that it may force some idempotents 
to be zero, which might eventually lead to the collapse of the whole algebra $D_{\underline{t}}(n, r)$, i.e., 
$D_{\underline{t}}(n, r)=0$. The following lemma shows what $\underline{t}$ is bad. 
For $\mu\in\Lambda(n, r)$, let $\mu s_i$ be the composition given by permuting the $i$-th and the $(i+1)$-th entries of $\mu$.

\begin{lemma}\label{well-def}
Suppose that $\lam=\mu s_i$ in $B(n,r)$ with $\lam_i=0$. If
 $t_{i, \lam}^{\lam_{i+1}}\neq t_{i, \mu}^{\lam_{i+1}}$ then $k_\lam=0$ or $k_\mu=0$ in $D_\ut(n,r)$.

\end{lemma}
\begin{proof}
By definition we get
$$k_\lam f_ie_ik_\lam=t_{i, \lam}k_\lam, \;\;k_\mu e_if_ik_\mu=t_{i, \mu}k_\mu $$
in $D_\ut(n,r)$.

Consequently, by the first equation in (\ref{cii}) we get
$$k_\lam f_i^{\lam_{i+1}}e_i^{\lam_{i+1}}k_\lam=t_{i, \lam}^{\lam_{i+1}}k_\lam,\;\;
k_\mu e_i^{\lam_{i+1}}f_i^{\lam_{i+1}}k_\mu=t_{i, \mu}^{\lam_{i+1}}k_\mu .$$

Hence
$$\aligned(k_\lam f_i^{\lam_{i+1}}e_i^{\lam_{i+1}}k_\lam)^2&=t_{i, \lam}^{2\lam_{i+1}}k_\lam\\
&=k_\lam  f_i^{\lam_{i+1}}e_i^{\lam_{i+1}} f_i^{\lam_{i+1}}e_i^{\lam_{i+1}}
k_\lam\\
&=k_\lam  f_i^{\lam_{i+1}}(k_\mu f_i^{\lam_{i+1}} e_i^{\lam_{i+1}}k_\mu) e_i^{\lam_{i+1}}
k_\lam\\
&=t_{i, \mu}^{\lam_{i+1}}k_\lam  f_i^{\lam_{i+1}}e_i^{\lam_{i+1}}
k_\lam\\
&=t_{i, \mu}^{\lam_{i+1}}t_{i, \lam}^{\lam_{i+1}}k_\lam,\\
\endaligned $$
i.e., \begin{equation}\label{def1}t_{i, \lam}^{\lam_{i+1}}(t_{i, \lam}^{\lam_{i+1}}-t_{i, \mu}^{\lam_{i+1}})k_\lam=0.
\end{equation}

Similarly,
$$\aligned(k_\mu e_i^{\lam_{i+1}}f_i^{\lam_{i+1}}k_\mu)^2&=t_{i, \mu}^{2\lam_{i+1}}k_\mu\\
&=t_{i, \mu}^{\lam_{i+1}}t_{i, \lam}^{\lam_{i+1}}k_\mu,\\
\endaligned $$
i.e., \begin{equation}\label{def2}t_{i, \mu}^{\lam_{i+1}}(t_{i, \mu}^{\lam_{i+1}}-t_{i, \lam}^{\lam_{i+1}})k_\mu=0.
\end{equation}

If
 $t_{i, \lam}^{\lam_{i+1}}\neq t_{i, \mu}^{\lam_{i+1}}$, then $k_\lam=0$ or $k_\mu=0$ by
 (\ref{def1}) and (\ref{def2}).
 This proves the lemma.

\end{proof}

\begin{remark}
 We want to avoid the case where $D_\ut(n,r)=0$. So from now on we only consider
 $\ut$ with
\begin{equation}\label{def t}
t_{\lam, i}=t_{\mu, i}, \; \text{if $\lam=\mu s_i$ \; \text{and }\; $\lam_i=0$}.
\end{equation}
 \end{remark}


Next we consider when the two algebras $D_\ut(n,r)$ and $D_{\underline{s}}(n,r)$ are isomorphic.
Suppose that $\ut$ and $\underline{s}$ satisfy condition (\ref{def t}) and $t_{i,\lam}s_{i,\lam}\neq 0$ for all $1\leq i\leq n-1$ and $\lam\in B(n,r)$.
 Define a map $\Phi$:
$$D_{\underline{t}}(n,r)\longrightarrow D_{\underline{s}}(n,r) $$ given by
$$k_\lam \mapsto k_\lam, \;f_{i,\lam}\mapsto f_{i,\lam}
\text{ and } e_{i,\lam}\mapsto \frac{t_{i, \lam}}{s_{i, \lam}}e_{i,\lam},
$$
and define a map $\Psi$:
$$D_{\underline{s}}(n,r)\longrightarrow D_{\underline{t}}(n,r)$$ given by
$$k_\lam \mapsto k_\lam, \;f_{i,\lam}\mapsto f_{i,\lam}
\text{ and } e_{i,\lam}\mapsto  \frac{s_{i, \lam}}{t_{i, \lam}}e_{i,\lam}.
$$

 For $\underline {a}=(a_1, a_2, \ldots, a_{n-1})\in\mathbb{F}^{n-1}$,
we write $\ut= \underline{a} \,\underline{s}$ if $t_{i, \lam}=a_is_{i, \lam}$ for all $1\leq i\leq n-1$ and $\lam\in B(n,r)$.
We have the following proposition.

\begin{proposition}\label{dimension}
The two algebras $D_\ut(n,r)\cong D_{\underline{s}}(n,r)$ are isomorphic  via the maps $\Phi$ and $\Psi$
 defined above if and only if 
$ \ut=\underline{a}\,\underline{s}$, where $\underline{a}\in \mathbb{F}^{n-1}$ and $a_i\neq 0$ for all $1\leq i\leq n-1$.
\end{proposition}
\begin{proof} Suppose that the two algebras are isomorphic via  the maps $\Phi$ and $\Psi$. 
As  $\Phi$ and $\Psi$ are algebra homomorphisms, the images of
generators should satisfy the generating relations of $D_{\underline{s}}(n,r)$ and $D_\ut(n,r)$, respectively.
By definition, when $\lam_i\lam_{j+1}\neq 0$ and $i\ne j$, we have the following equation.
$$\aligned  \Phi(e_if_jk_\lam-f_je_ik_\lam)&=\frac{t_{i, \lam-\alpha_j}}{s_{i, \lam-\alpha_j}}e_if_jk_\lam-\frac{t_{i, \lam}}{s_{i, \lam}}f_je_ik_\lam\\
&=(\frac{t_{i, \lam-\alpha_j}}{s_{i, \lam-\alpha_j}}-\frac{t_{i, \lam}}{s_{i, \lam}})e_if_jk_\lam\\
&=0.
\endaligned$$
This implies that \begin{equation}\label{rel 1}\frac{t_{i, \lam-\alpha_j}}{s_{i, \lam-\alpha_j}}=\frac{t_{i, \lam}}{s_{i, \lam}}\end{equation}
for $\lam\in B(n,r)$ with $\lam_i\lam_{j+1}\neq 0$ and $i\ne j$.

Similarly, when $\lam_i\lam_{i+1}\neq 0$, we have the following equation.
$$\aligned\Phi(e_if_ik_\lam-f_ie_ik_\lam)&=\frac{t_{i, \lam-\alpha_i}}{s_{i, \lam-\alpha_i}}e_if_ik_\lam-\frac{t_{i, \lam}}{s_{i, \lam}}f_ie_ik_\lam\\
&=(\frac{t_{i, \lam-\alpha_i}}{s_{i, \lam-\alpha_i}}-\frac{t_{i, \lam}}{s_{i, \lam}})e_if_ik_\lam\\
&=0.
\endaligned$$
This implies that \begin{equation}\label{rel 2}\frac{t_{i, \lam-\alpha_i}}{s_{i, \lam-\alpha_i}}=\frac{t_{i, \lam}}{s_{i, \lam}}\end{equation}
for $\lam\in B(n,r)$ with $\lam_i\lam_{i+1}\neq 0$. So 
$$ \ut=\underline{a}\,\underline{s},$$
as stated. 

On the other hand, supposed that $\ut=\underline{a}\underline{s}$ for some  $\underline{a}\in \mathbb{F}^{n-1}$ with all $a_i\not=0$. By $(10)$ and $(11)$, we know that $\Phi$ is compatible with the commutative relations. Further, 
when $i=j-1$, $\lam_{i+1}\geq 2$ and $\lam_{j+1}\geq 1$, the following
Serre relation can be deduced from (\ref{rel 1}) and (\ref{rel 2}).
$$\aligned\Phi(e_i^2e_jk_\lam-e_ie_je_ik_\lam)&=\frac{t_{i, \lam+\alpha_i+\alpha_j}}{s_{i, \lam+\alpha_i+\alpha_j}}\frac{t_{i, \lam+\alpha_j}}{s_{i, \lam+\alpha_j}}\frac{t_{j, \lam}}{s_{j, \lam}}e_i^2e_jk_\lam-\frac{t_{i, \lam+\alpha_i+\alpha_j}}{s_{i, \lam+\alpha_i+\alpha_j}}\frac{t_{j, \lam+\alpha_i}}{s_{j, \lam+\alpha_i}}
\frac{t_{i, \lam}}{s_{i, \lam}}e_ie_je_ik_\lam\\
&=\frac{t_{i, \lam+\alpha_i+\alpha_j}}{s_{i, \lam+\alpha_i+\alpha_j}}(\frac{t_{i, \lam+\alpha_j}}{s_{i, \lam+\alpha_j}}\frac{t_{j, \lam}}{s_{j, \lam}}-\frac{t_{j, \lam+\alpha_i}}{s_{j, \lam+\alpha_i}}
\frac{t_{i, \lam}}{s_{i, \lam}})e_i^2e_jk_\lam\\
&=0.
\endaligned$$

Similarly we can prove that the Serre relations on $f_i$ can be deduced from (\ref{rel 1}) and (\ref{rel 2}).
And the relations $C_{ij, \lam}(\ut)$ can be proved directly without requirements on $t_{i,\lam}$ and $s_{\lam, i}$.
So $\Phi$ is a homomorphism of algebras with $\Psi$ the inverse, and so an isomorphism. 
This proves the proposition.
\end{proof}

\begin{corollary}\label{cor4.7}
Suppose that $t_{i,\lam}=t_i\neq 0$ for all $\lam\in B(n,r)$ and $1\leq i\leq n-1$. Then
$D_{\underline{t}}(n,r)\cong S_0(n,r)$.
\end{corollary}

\section{ The degeneration $DS_0(n, r)$ of $S_0(n,r)$}\label{degalg}

\subsection{A new algebra defined on double flag varieties}
Let $DS_0(n,r)$ be the $\mathbb{F}$-space with basis $\{e_A\mid A\text{ is in } \Xi(n,r)\}$. Define a multiplication in $DS_0(n, r)$
as follows,
$$e_{A}\star e_{B}=
\left\{
\begin{tabular}{ll}
$e_A\cdot e_B$ & if $E(e_A)+E(e_B)=E(e_A\cdot e_B);$
\\
$0$ & otherwise.
\end{tabular}
\right.$$
In other words, by Lemma \ref{subadd} and Corollary \ref{cor1}, $DS_0(n,r)$ is the associated graded algebra of
the filtered algebra $S_0(n,r)$. We often skip the multiplication signs in case of no confusion.

We give a geometric interpretation of the condition in the definition of $\star$.
Suppose that $$e_A=[f, g], \; e_B=[g, h], \; \mbox{ and } e_C=[f', h']$$
with $$e_A\cdot e_B=e_C$$ in $S_0(n, r)$.
Recall that $V$ is an $r$-dimensional vector space defined over a field $k$.
We view an $n$-step flag in $V$ as a representation of a linear quiver $A_n$, where vertex $1$ is a source and vertex $n$ is
a sink. Denote the indecomposable projective representation of $A_n$ by $P_i$.



\begin{lemma}\label{lemma4.1}
$\dvector\, f'\cap h'\geq \dvector\, f\cap g +\dvector\, g\cap h-\dvector\, g. $
Consequently, $$\dvector\, f'\cap h'=\dvector\, f\cap g +\dvector\, g\cap h-\dvector\, g $$
if and only if $$E(e_A)+E(e_B)=E(e_A\cdot e_B).$$
\end{lemma}

\begin{proof}
Note that $$E(e_A)=\dvector f -\dvector f\cap g, \,E(e_B)=\dvector g -\dvector h\cap g, \,E(e_Ae_B)=\dvector f' -\dvector f'\cap h'$$
and $f\cong f'$ as representations.
Now the lemma follows from Lemma \ref{subadd}.
\end{proof}

Therefore the product $\star$ can be defined as below.

\begin{lemma}\label{geominterpretation}
$$[f, g]\star [g, h]=
\left\{
\begin{tabular}{ll}
$[f', h']$ & if  $\dvector\, f'\cap h'=\dvector\, f\cap g +\dvector\, g\cap h-\dvector\, g$,
\\
$0$ & otherwise.
\end{tabular}
\right.$$
\end{lemma}

\begin{example}\label{newex}
Let $A=\left( \begin{tabular}{ll} 2 &2\\0&1 \end{tabular}\right)$,  $B=\left( \begin{tabular}{ll} 2 &0\\2&1 \end{tabular}\right)$ and
 $C=\left( \begin{tabular}{ll} 3 &1\\1&0 \end{tabular}\right)$. Then
$$e_A e_B=e_C,$$ but
$$ e_A\star e_B=0.$$
\end{example}

By Lemma \ref{descriptionofideals},  the $e_C$ in Example \ref{newex} is contained in the ideal $I(2, 5)$.
This indicate that there should be a link between $DS_0(n, r)$ and the quotient algebra $S_0(n, r)/I(n, r)$.
We will explore the relation and their relation to $D_0(n, r)$ in the remaining part of this section. 

\begin{lemma}\label{generator}
$DS_0(n,r)$ is generated by $e_{i,\lambda}$, $f_{i,\lambda}$ and $k_\lambda$
for $\lambda\in \Lambda(n,r)$ and $1\leq i\leq n-1$.
\end{lemma}
\begin{proof}
The lemma follows from the definition of multiplication in $DS_0(n,r)$ and
the construction of the basis elements in Lemma \ref{monomial}.
\end{proof}

Recall the algebra $\mathbb{F}\Sigma(n,r)/\mathcal{J}(\underline{t})$ in Section 4. In particular, we have 
 $$D_{\underline{0}}(n, r)=\mathbb{F}\Sigma(n,r)/\mathcal{J}(\underline{0}).$$ 
The difference between $D_{\underline{0}}(n, r)$ and $S_0(n, r)$ is that the idempotent relations 
in $S_0(n, r)$ are replaced by zero relations. 
Recall also for $j\geq i$,
$$e(i,j)=e_i\cdot e_{i+1}\cdot ... \cdot e_j.$$

\begin{lemma}\label{comm}
Let $i<j$ and $a,b\geq 0$ such that $j-i>a+b$. Then
$$e(i,j)\cdot e(i+a,j-b)=e(i+a,j-b)\cdot e(i,j)$$
in $D_{\underline{0}}(n,r)$.
\end{lemma}
\begin{proof}
It suffices to prove that $e(i,j)\cdot e_l = e_l\cdot e(i,j)$ for all $l=i,\cdots,j$. In fact,
due to the commutativity relation $e_me_{m'}=e_{m'}e_m$ when $|m-m'|>1$, we need only prove
$e(i,j)\cdot e_l=e_l\cdot e(i,j)$ for $(i,j)=(l-1,l),(l-1,l+1),(l,l+1)$. All the three cases follow
 from the relations $P_{mm'}$.
\end{proof}

\begin{proposition}\label{dimensionlemma}
$\mathrm{dim}_\mathbb{F}D_{\underline{0}}(n,r)=
\mathrm{dim}_{\mathbb{F}}S_0(n,r)$

\end{proposition}
\begin{proof}
 Observe that in both algebras, a complete list of representatives of nonzero paths is a basis. In particular,  the paths of the form
in Lemma \ref{monomial} (2), which are all reduced paths, form  a basis for $S_0(n, r)$.

As the relations $P_{ij}$ and $N_{ij}$ are the same in both algebra, if two nonzero paths are
the same in $D_{\underline{0}}(n,r)$, they are then the same in $S_0(n, r)$. Further,
 any reduced path that is nonzero in $\zschur$ is also nonzero in
$D_{\underline{0}}(n,r)$. Indeed,
take a nonzero reduced path $\rho=...e_{i, \lambda}..f_{j, \mu}...$ in
$\Sigma(n,r)$.
Now suppose that $0=\rho\in D_{\underline{0}}(n,r)$. This implies that
using relations $P_{ij, \lambda}, N_{ij, \lambda}$,
$$\rho=...e_{i, \lambda-\alpha_i}f_{i, \lambda}{k_\lam}...\mbox{ or }...{k_\lam}f_{i, \lambda+\alpha_i}e_{i, \lambda}... \;\;\; (\dagger)$$
with $\lambda$ at the boundary. So the new expression $(\dagger)$ also holds  in $\zschur$.
By the relations
$C_{ij, \lambda}$, $$e_{i, \lambda-\alpha_i+\alpha_{i+1}}f_{i, \lambda}{k_\lam}=
{k_\lam}f_{i, \lambda+\alpha_i-\alpha_{i+1}}e_{i, \lambda}=k_\lam,$$ which
contradicts the minimality of the number of arrows in $\rho$. So $\rho$
is a non-zero path
 in $D_{\underline{0}}(n,r)$. Therefore
\begin{equation}\label{inequal}\mathrm{dim}_\mathbb{F}D_{\underline{0}}(n,r) \geq \mathrm{dim}_{\mathbb{F}}S_0(n,r).
\end{equation}

Next we claim that any reduced path $\rho$  in $D_{\underline{0}}(n,r)$ is equal to a path of the form $(2)$ in
Lemma \ref{monomial} and thus by the observation and the inequality (\ref{inequal}), 
the dimensions of the two algebras are the same.
We proceed the proof of the claim by induction on the length of $\rho$. When the length of $\rho$ is at most
$1$, then it already has the required form. Assume that the length is larger than $1$, and
that $\rho = \rho\cdot k_\lambda$.
We have $$\rho = \rho'e_ik_\lambda\mbox{ or }\rho'f_ik_\lambda$$ where $\rho'$ is a path
of length one less than $\rho$. By the induction hypothesis, we may write $\rho'$
of the form (2) in Lemma \ref{monomial},
$$\rho'=EFk_{\lambda'}$$ where $E=e(i_s,j_s)\cdot ....\cdot e(i_1,j_1)$ and $F=f(i'_1,j'_1)\cdot ... \cdot
f(i'_t,j'_t)$.
We first consider the case $\rho = \rho'e_ik_\lambda$. Since $\rho$ is reduced, we have
$$\rho = EFe_ik_\lambda = Ee_iFk_\lambda.$$

If $j_m\neq i-1$ for all $m$, the relations $P_{ij}$ 
give us the required form
$$Ee_iFk_\lambda=e(i_s,j_s)\cdot ....\cdot e(i_l,j_l) \cdot e(i,i) \cdot ... \cdot e(i_1,j_1)Fk_\lambda,$$ where
$l$ is the smallest integer such that $j_l>i$.

Let $m$ be the smallest integer with $j_m=i-1$. 
We have
$$Ee_i=e(i_s,j_s)\cdot ....\cdot e(i_m,j_m+1) \cdots e(i_1,j_1).$$
By repeatedly applying Lemma \ref{comm} and relations $P_{ij}$, we can move $e(i_m,j_m+1)$ to the
appropriate position.

The case $\rho=\rho'f_ik_\lambda$ is similar, and so we skip the details.
In either case, the path $\rho$ is equal to a path of the form (2) in Lemma \ref{monomial}, as claimed.
\end{proof}

\begin{remark}
By Corollary \ref{cor4.7} and Proposition \ref{dimensionlemma}, 
the algebra  $D_{\underline{0}}(n, r)$ is a degeneration of 
the $0$-Schur algebra $S_0(n, r)$. 
\end{remark}
The following is the main result of this paper.


\begin{theorem} \label{maintheorem}
We have isomorphisms 
$$
S_0(n,r+n)/I(n,r+n)\cong D_{\underline{0}}(n,r)\cong DS_0(n,r),
$$
as $\mathbb{F}$-algebras.
\end{theorem}
\begin{proof}
We first prove the first isomorphism.
Embed $\Sigma(n, r)$ into the interior of $\Sigma(n, r+n)$ via $\phi: k_{\lambda}\mapsto k_{\mu}$ and embed
the arrows correspondingly, where $\mu=(\lam_1+1, \lam_2+1, \ldots, \lam_n+1)\in \Lambda(n,r+n)$. Then observe that the relations $\phi(P_{ij, \lambda})$,
$\phi(N_{ij, \lambda})$ and $\phi(C_{ij, \lambda}(\underline{0}))$ hold in
$S_0(n, r+n)/I(n,r+n)$. So we have a
surjective map
 $$D_{\underline{0}}(n,r)\twoheadrightarrow S_0(n, r+n)/I(n,r+n)$$
By Corollary \ref{dim1},
$$
\dim_\mathbb{F} S_0(n, r+n)/I(n,r+n)=\left(\begin{tabular}{c}$n^2+(n+r)-n-1$ \\$(n+r)-n$\end{tabular} \right)=\left(\begin{tabular}{c}$n^2+r-1$ \\$r$\end{tabular} \right),$$
which is the dimension of $S_0(n, r)$. So
 by Proposition \ref{dimensionlemma},
$$\dim_\mathbb{F} S_0(n, r+n)/I(n,r+n)= \dim_\mathbb{F} D_{\underline{0}}(n,r).$$
Thus the two algebras are isomorphic.

Mapping the vertices $k_\lambda$ and the arrows $e_{i,\lambda}$ and $f_{i,\lambda}$
to the corresponding generators of $DS_0(n,r)$ defines an algebra homomorphism
$\psi$ from $D_{\underline{0}}(n,r)$ to $DS_0(n,r)$.
By Lemma \ref{generator}, $\psi$ is an epimorphism. Note that
by definition,
$$\dim_\mathbb{F} DS_0(n, r)=\dim_\mathbb{F} S_0(n, r)$$
and thus by Proposition \ref{dimensionlemma},
$$\dim_\mathbb{F} DS_0(n, r)=\dim_\mathbb{F} D_{\underline{0}}(n,r).
$$
Therefore, $\psi$ is an isomorphism and thus the three algebras are isomorphic as claimed.
\end{proof}

\begin{remark}
\begin{itemize}
\item[(1)] By the isomorphisms in  Theorem \ref{maintheorem}, we can view 
any of the three algebras as a degeneration of $S_0(n, r)$.

\item[(2)] The multiplication $\star$ in Lemma \ref{lemma4.1} gives a geometric interpretation of the multiplications 
in $DS_0(n, r)$. Thus the space of $\GL(V)$-orbits in the double flag varieties $\mathcal{F}\times \mathcal{F}$ 
associated with the multiplication $\star$ gives a geometric construction of the three algebras in Theorem \ref{maintheorem}.
\end{itemize}
\end{remark}

\section{ centralizer algebras of $ D_\ut(n,r)$}\label{degalg}

In this section, we discuss relations between  the $0$-Hecke algebra $H_0(r)$, the nil-Hecke algebra $NH_0(r)$,
the nil-Temperley-Lieb algebra $NTL(r)$ and
degenerate $0$-Schur algebras. Throughout this section, we let $\alpha=(1, \dots, 1)$ be the composition in
$\Lambda(r, r)$ with all the entries equal to $1$.

\subsection{$H_0(r)$, $NH_0(r)$ and $DS_0(r, r)$}\label{sect5.1}
We first recall some definitions and key facts. We then show that the associated graded algebras of $H_0(r)$ and
$NH_0(r)$, with the filtration induced by the degree function
discussed in Section \ref{grading}, are isomorphic to a subalgebra of
$k_\alpha DS_0(r, r)k_\alpha$.

\begin{definition}[\cite{Dudeng},  \cite{Thibon}] \label{d2}  The 0-Hecke algebra, denoted by $H_0(r)$, is the $\mathbb{F}$-algebra
 generated by $T_1, T_2, \ldots, T_{r-1}$ with defining
relations
\begin{equation}\label{re1}
\begin{cases} T_i^2=T_i, \; \;& \text{for} \;1\leq i\leq r-1;\\
T_iT_j=T_jT_i, \;\;&\mbox{for $1\leq i,j\leq r-1$ with $|i-j|>1$};\\
T_iT_{i+1}T_i =T_{i+1}T_iT_{i+1},\;\;&\text{for $1\leq i\leq
r-2$}.
\end{cases}
\end{equation}
\end{definition}


\begin{theorem} [Theorem 10.4, \cite{JS}] \label{thmjs}
The algebras $H_0(r)$ and $k_\alpha S_0(r, r)k_\alpha$ are isomorphic via the the map
$T_i\mapsto k_\alpha f_ie_i k_\alpha$.
\end{theorem}

\begin{definition}[\cite{Roq}] \label{d1}
The nil-Hecke algebra, denoted by $NH_0(r)$, is the unital $\mathbb{F}$-algebra
 generated by $T_1, T_2, \ldots, T_{r-1}$ with defining
relations
\begin{equation}\label{re1}
\begin{cases} T_i^2=0, \; \;& \text{for} \;1\leq i\leq r-1;\\
T_iT_j=T_jT_i, \;\;&\mbox{for $1\leq i,j\leq r-1$ with $|i-j|>1$};\\
T_iT_{i+1}T_i =T_{i+1}T_iT_{i+1},\;\;&\text{for $1\leq i\leq
r-2$}.
\end{cases}
\end{equation}
\end{definition}

Assume that $w=s_{i_1}\cdots s_{i_t}=s_{j_1}\cdots s_{j_t}$ are
reduced expressions in the symmetric group $S_r$ on $r$ letters,
where $s_i$ is the transposition $(i, i+1)$.
As reduced expressions of $w$ can be obtained from each other by
using the braid relations only (see \cite{Matsumoto, tits}), we have 
 $$T_{i_1}\cdots
T_{i_t}=T_{j_1}\cdots T_{j_t}$$ 
in both $H_0(r)$ and $NH_0(r)$,
and thus the element
$$T_w=T_{i_1}\cdots T_{i_t}$$
is well-defined in both algebras. One can also deduce
that $\{T_w~|~w\in S_r\}$ is a basis for both algebras. Further,

\begin{equation}\label{0Heckemult}
\mbox{  in } H_0(r), \;
T_{i}T_{w}=
\begin{cases}
T_{s_iw},& \ell(s_iw)=\ell(w)+1;\\
T_w,  & \text{otherwise},
\end{cases}
\end{equation}
and
\begin{equation}\label{Heckemult}
\mbox{  in } NH_0(r), \;
T_{w_1}T_{w_2}=
\begin{cases}
T_{w_1w_2},& \ell(w_1w_2)=\ell(w_1)+\ell(w_2);\\
0,  & \text{otherwise},
\end{cases}
\end{equation}
 where $\ell: W\rightarrow\bbn\cup\{0\}$ is the length function of elements in $S_r$.

By Theorem \ref{thmjs}, the filtration   of the $0$-Schur algebra $S_0(r, r)$
discussed in Section \ref{grading} induces a filtration  on $H_0(r)$. When the length equation in (\ref{0Heckemult}) or
(\ref{Heckemult}) holds,
the multiplications of $T_w$ and $T_{w'}$ in  $H_0(r)$ and  $NH_0(r)$ are the same, so we also have a
filtration on $NH_0(r)$. Further,
together with the fact that  $DS_0(r, r)$ is the associated graded algebra of $S_0(r, r)$,
this implies the following.

\begin{proposition}
The associated graded algebras of  $H_0(r)$ and $NH_0(r)$ are isomorphic to the algebra  $k_\alpha DS_0(r, r) k_\alpha$.
\end{proposition}

\subsection{Nil-Temperley-Lieb algebras}\label{sec5.2}
The nil-Temperley-Lieb algebra  $NTL(r)$ is 
the quotient algebra of $NH_0(r)$ modulo the ideal generated by $T_iT_{i+1}T_i$ for $1\leq i\leq r-2$
(see e.g. \cite{Fomin}).

\begin{lemma}
The algebra ${NTL}(r)$ has a basis consisting of $T_w$, where  $w$ does not contain a subword of the form   $s_is_js_i$  in any of its reduced expressions,
for any $i$ and $j$ with $|i-j|=1$.
\end{lemma}


\begin{proof}
It follows from the fact that $\{T_w|w\in S_r\}$ is a basis of $NH_0(r)$ and the mulitplication in  (\ref{Heckemult}).
\end{proof}

Note that $\mathrm{dim}NTL(r)$ is the Catalan number $\frac{1}{r+1}\left(\begin{matrix} 2r\\ r\end{matrix}\right)$ and there is a useful combinatorial parametrization of the elements in a basis of  $NTL(r)$,
 using Dyck words.  Pictorially, Dyck words can be described using peak pictures in a  triangle with $r$ dots on each edge (cf \cite{Stanley}).
For instance, when $r=3$, the five peak pictures are as follows.

$$
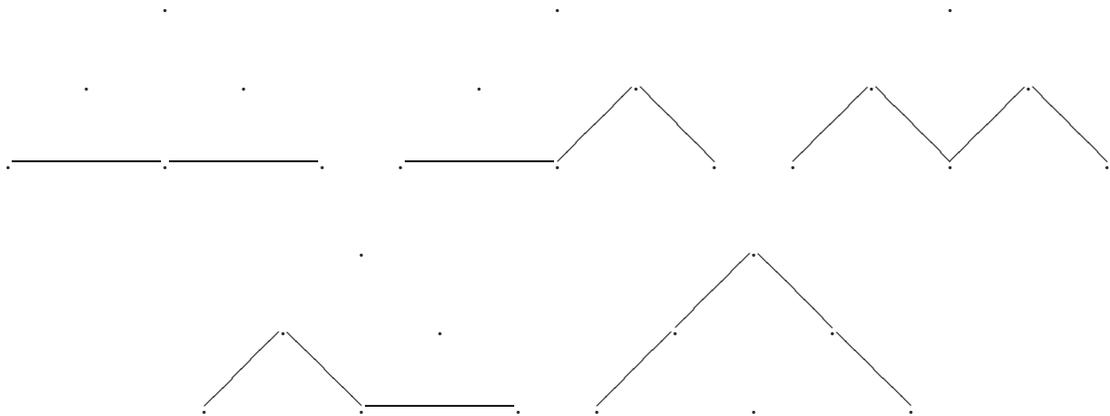
\xymatrix@!=0.2cm @M=0pt{
&&.&&                                               &   &&.&&                                           &   &&.&&
\\&.&&.&                                           &   &.&&.&                                          &   &.&&.&
\\.\ar@{-}[rr]&&.\ar@{-}[rr]&&.      &   .\ar@{-}[rr]&&.\ar@{-}[ur]&&. \ar@{-}[ul]   &    .\ar@{-}[ur]&&.\ar@{-}[ul]\ar@{-}[ur]&&.\ar@{-}[ul]
} $$

\vspace{5mm}
$$
\xymatrix@!=0.2cm @M=0pt{
&&.&&                                                          &   &&.&&
\\&.&&.&                                                        &   &.\ar@{-}[ur]&&.\ar@{-}[ul]&
\\.\ar@{-}[ur]&&.\ar@{-}[ul]\ar@{-}[rr]&&.         &   .\ar@{-}[ur]&&.&&. \ar@{-}[ul]
} $$

\vspace{7mm}

Let
$$x_i=k_\alpha f_ie_ik_\alpha\in k_{\alpha}DS_0(n,r)k_{\alpha}.$$
Denote by
$\underline{DS}(r, \alpha)$ the subalgebra of $  k_{\alpha} DS_0(r,r) k_{\alpha}$ generated by $x_i$ for $1\leq i\leq r-1$.
When $r=3$, the algebra $\underline{DS}(r, \alpha)$  is five dimensional, the orbit basis elements are determined by the following matrices.
\vspace{3mm}
$$
\left(
\begin{matrix}
1 & 0&0 \\
0& 1&0\\
0&0&1\\
\end{matrix}
\right),
\left(
\begin{matrix}
1 & 0&0 \\
0& 0&1\\
0&1&0\\
\end{matrix}
\right),
\left(
\begin{matrix}
0 & 0&1 \\
1& 0&0\\
0&1&0\\
\end{matrix}
\right),
\left(
\begin{matrix}
0 & 1&0 \\
1& 0&0\\
0&0&1\\
\end{matrix}
\right),
\left(
\begin{matrix}
0 & 1&0 \\
0& 0&1\\
1&0&0\\
\end{matrix}
\right).\\
$$

\vspace{3mm}

Similar to  $T_w$, $$x_w=x_{i_1}\dots x_{i_t}$$ is well-defined, where $s_{i_1}\dots s_{i_t}$ is a reduced expression of $w$.
By direct computation following  the fundamental multiplication rules and  the definition of $DS_0(n, r)$,
we have the following lemma.

\begin{lemma}\label{TLsurj}
The elements  $x_i$ for $1\leq i\leq r-1$ satisfy the generating relations of ${NTL}(r)$.
Consequently,  there is an epimorphism
$$ {NTL}(r)\longrightarrow  \underline{DS}(r,\alpha), \;\; T_i\mapsto x_i, \;\;\forall i.$$
\end{lemma}

Any matrix that determines a nonzero orbit in  $\underline{DS}(r, \alpha)$ is a permutation matrix. We call an nonzero entry
that is below the
diagonal a peak entry. For a peak entry $(i, j)$, which implies that $j<i$, we call the entries $(j, j)$ and $(i, i)$ the feet of the peak.
For the five matrices above, which determine the orbit basis  elements in $\underline{DS}(r, \alpha)$, we connect the peaks, feet and
diagonal $1$-entry, using zig-zag lines as follows.
\vspace{3mm}
$$
\xymatrix@!=0.05cm @M=0pt{
1 \ar@{-}[dr]& 0&0 \\
0& 1\ar@{-}[dr]&0\\
0&0&1 \\
} \;\;\;\;\;
\xymatrix@!=0.05cm @M=0pt{
1\ar@{-}[dr] & 0&0 \\
0& 0\ar@{-}[d]&1\\
0&1\ar@{-}[r]&0\\
}\;\;\;\;\;
\xymatrix@!=0.05cm @M=0pt{
0\ar@{-}[d] & 0&1 \\
1\ar@{-}[r]& 0\ar@{-}[d]&0\\
0&1\ar@{-}[r]&0\\
}
\;\;\; \;\;
\xymatrix@!=0.05cm @M=0pt{
0\ar@{-}[d] & 1&0 \\
1\ar@{-}[r]& 0\ar@{-}[dr]&0\\
0&0&1\\
}\;\;\;\;\;
\xymatrix@!=0.05cm @M=0pt{
0 \ar@{-}[dd]& 1&0 \\
0& 0&1\\
1\ar@{-}[rr]&0&0\\
}
$$

\vspace{3mm}

\noindent In this way,  we obtain a well defined one-to-one correspondence between the  basis elements of $NTL(3)$
and $\underline{DS}(3, \alpha)$.
In fact, such a one-to-one correspondence exists for all $r$.

\begin{theorem}\label{NTL}
The two algebras
${NTL}(r)$ and $\underline{DS}(r,\alpha)$  are isomorphic.
\end{theorem}

\begin{proof}
First note the nonzero orbits $e_A$ in $\underline{DS}(r, \alpha)$ form a basis and each orbit $e_A$ is uniquely determined by
the matrix $A$.
By Lemma \ref{TLsurj}, there is  a surjective morphism from ${NTL}(r)$ to $\underline{DS}(r,\alpha)$. So it suffices to show that there exists a nonzero element generated by the $x_i$s such that the corresponding matrix
produces the peak picture.

First identify the triangle, in which  we draw the peak pictures,  with the lower triangular part of an $r\times r$ matrix.
The peaks give us peak entries, $(i_1, j_1), \dots, (i_s, j_s)$ where for any $l$,
$$ j_l<i_l <i_{l+1} \mbox{ and } j_l<j_{l+1}.$$
 Let $$
x^{(l)}=x_{i_l-1}\dots x_{j_l+1}x_{j_l} \mbox{ and } x= x^{(1)}\dots x^{(s)}.
$$

By the fundamental multiplication rules,
the multiplication of $x_{i+1}$ with  $x_{i}\dots x_{j_s} k_\alpha$ moves  a $1$ on or above  
the diagonal  further up in the same column
and a $1$ on or below the diagonal further down in the same column and thus
$$ E(x_{i+1})+E(x_{i}\dots x_{j_s} k_\alpha )=E(x_{i+1}x_{i}\dots x_{j_s}k_\alpha).$$
That is, the equality in the definition of $\star$ holds.
Therefore,
$$x^{(s)} k_\alpha\not=0.$$
Similarly,  $$x^{(l)}( x^{(l+1)}\dots x^{(s)} k_\alpha)\not=0$$
and when compared with $ x^{(l+1)}\dots x^{(s)} k_\alpha$, it has a new  peak at $(i_l, j_l)$.
Thus we have found a nonzero element $x=xk_\alpha$ such that the corresponding matrix gives us the required peak picture.
\end{proof}

\begin{example}
In this example we demonstrate the process of obtaining the matrix with two peaks above 
Theorem \ref{NTL} 
using the construction in the proof of the theorem.
The peak picture  has two peak entries $(2, 1)$ and $(3, 2)$. By definition, 
$$x^{(2)}=x_2 \mbox{ and } x^{(1)}=x_1.$$
Then $x^{(2)}k_\alpha$ and  $x^{(1)}x^{(2)}k_\alpha$ are the orbit basis elements corresponding to the matrices, respectively,  

$$\left( \begin{matrix}
1& 0&0 \\
0& 0&1\\
0&1&0\\
\end{matrix}\right)  \quad \text{ and } \quad
 \left(\begin{matrix}
0 &0 &1 \\
1& 0&0\\
0&1&0\\
\end{matrix}\right).
$$
\end{example}

We have the following special version of Lemma \ref{geominterpretation}.

\begin{lemma} \label{geominterpretation2} Use the same notation Lemma \ref{geominterpretation}. Further, assume that all the flags are isomorphic to
$\oplus_{i=1}^nP_i$ and $[f, g]=x_i$ for some $i$. Then the following is true in $\underline{DS}(r, \alpha)$.
$$[f, g]\star [g, h]=
\left\{
\begin{tabular}{ll}
$[f', h']$ & if  $\dim f'\cap h'=\dim g\cap h-1 $,
\\
$0$ & otherwise.
\end{tabular}
\right.$$
\end{lemma}

\begin{proof}
Note that as a representation, $g/(f\cap g)$ is isomorphic to the simple representation of the linear quiver $A_n$ at vertex $i$,
since $[f, g]=x_i$. Now the lemma follows from Lemma \ref{lemma4.1} and \ref{geominterpretation}.
\end{proof}

\begin{remark} In the light of Lemma \ref{geominterpretation2},  Theorem \ref{NTL} gives a geometric realisation of nil-Temperley-Lieb algebras, via double 
flag varieties.
\end{remark}

\subsection{An observation}

Using the presentation of $S_0(r, r)$ in Section \ref{presentation}, we define a new algebra $\hat{D}S_0(r, r)$  as a quotient algebra of
$\mathbb{F}\Sigma(r, r)$. This new algebra has the same relations as $S_0(r, r)$ except that the idempotent
relations $$e_{i, \alpha}f_{i, \lam}=k_{\lam},  \quad f_{i, \alpha}e_{i, \mu}=k_{\mu}$$ are replaced by
\begin{equation}\label{hat}e_{i, \alpha}f_{i, \lam}=0 \mbox{ and } f_{i, \alpha}e_{i, \mu}=0,\end{equation}
for $1\leq i\leq r-1$, where $\lam=\alpha+\alpha_i$, $\mu=\alpha-\alpha_i$, $\alpha$ is the composition in
$\Lambda(r, r)$ with all the entries equal to $1$ and $\alpha_i$ is defined in Section \ref{presentation}.
The new algebra $\hat{D}S_0(r, r)$ has fewer zero relations coming from
idempotent relations in $S_0(r, r)$ than $D_{\underline{0}}(r, r)$. We only force those idempotent relations that go via the 
centre $k_{\alpha}$ of the simplex $\Sigma(n, r)$ to be $0$. 
Let
$$ y_i=e_{i, \alpha-\alpha_i}f_{i, \alpha}.$$Note that by the multiplication rules, 
$$ y_i=e_{i}f_{i, \alpha}= e_{i, \alpha-\alpha_i}f_{i},$$
because once the starting or ending idempotent is given, then the 
product of a sequence of $e_i$s and $f_j$s is determined, we don't need to indicate the composition indices  
$\lambda$ for the other factors.

\begin{lemma}\label{lastlem}
The elements $y_i$, where $1\leq i\leq r-1$, satisfy the generating relations of $NTL(r)$.
\end{lemma}

\begin{proof}
It can be checked directly that when $|i-j|>1$, $y_iy_j=y_jy_i$ for $1\leq i, j\leq r-1$. 
So it suffices to prove that $y_iy_{i+1}y_i=y_{i+1}y_iy_{i+1}=0$. Since the proof is similar, we 
only prove that $y_iy_{i+1}y_i=0$. 

For $1\leq i\leq r-1$, ${\alpha-\alpha_i+\alpha_{i+1}}$ is of the form $(1,\ldots,1,0,3,0,1,\ldots,1)$, 
which has  $3$ at the $(i+1)$th entry.
We first prove that $k_{\alpha-\alpha_i+\alpha_{i+1}}=0$. 
We have
$$
\begin{aligned}k_{\alpha-\alpha_i+\alpha_{i+1}}&=
f_{i}^2e_{i}^2 e_{i+1}f_{i+1} k_{\alpha-\alpha_i+\alpha_{i+1}}\\
&=f_{i}^2  e_{i} e_{i+1}e_{i}f_{i+1} k_{\alpha-\alpha_i+\alpha_{i+1}}\\
&=f_{i}^2  e_{i} e_{i+1}f_{i+1} e_{i}k_{\alpha-\alpha_i+\alpha_{i+1}}\\
&=0,
\end{aligned}$$
by (\ref{hat}).
Consequently, 
$$
\begin{aligned}y_iy_{i+1}y_i&=
e_{i}f_{i}  f_{i+1}  e_{i+1} e_{i}f_{i, \alpha}\\
&=f_{i} e_{i}   f_{i+1}  e_{i+1} e_{i}f_{i, \alpha}\\
&=f_{i}  f_{i+1}   e_{i} e_{i+1} e_{i}f_{i, \alpha}\\
&=f_{i}  f_{i+1}   e_{i}  e_{i}  e_{i+1} f_{i, \alpha}\\
&=f_{i}  f_{i+1}   e_{i}  e_{i} k_{\alpha-\alpha_i+\alpha_{i+1}} e_{i+1} f_{i, \alpha}\\
&=0.
\end{aligned}$$
This proves the lemma. 
\end{proof}

\begin{example}[An observation] 
$$
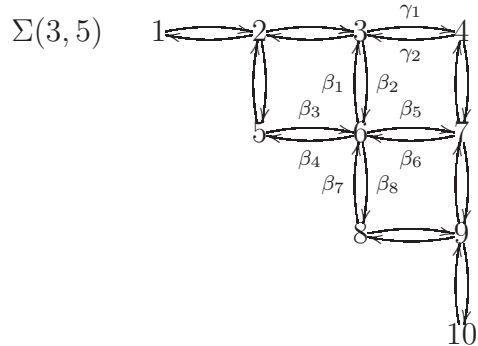
\xymatrix@!=0.5cm @M=0pt{
&&&& \Sigma(3, 5)&1 \ar@/^0.2pc/[r]&\ar@/^0.2pc/[r]2\ar@/^0.2pc/[d]\ar@/^0.2pc/[l]
 \ar@/^0.2pc/[r]&3\ar@/^0.2pc/[r]^{\gamma_1}\ar@/^0.2pc/[l]\ar@/^0.2pc/[d]^{\beta_2}
 &4\ar@/^0.2pc/[l]^{\gamma_2}\ar@/^0.2pc/[d]&&&&&\\
 &&
 &&&&5\ar@/^0.2pc/[u]\ar@/^0.2pc/[r]^{\beta_3}&\ar@/^0.2pc/[u]^{\beta_1}6\ar@/^0.2pc/[d]^{\beta_8}\ar@/^0.2pc/[l]^{\beta_4}\ar@/^0.2pc/[r]^{\beta_5}
&7\ar@/^0.2pc/[l]^{\beta_6}\ar@/^0.2pc/[u]\ar@/^0.2pc/[d]&&&&\\
 &&&&&&&8\ar@/^0.2pc/[r]\ar@/^0.2pc/[u]^{\beta_7}&9\ar@/^0.2pc/[d]\ar@/^0.2pc/[l]\ar@/^0.2pc/[u]&&&&&\\
 &&&&&&&&10\ar@/^0.2pc/[u]&&\\
 &&&&&&&&& &&\\
  }$$
  The composition of two paths is written in the form such as $\beta_3: 5\longrightarrow 6$, $\beta_5: 6\longrightarrow 7$,
  then $\beta_5\beta_3: 5\longrightarrow 7$.
 The algebra
 $\hat{D}S_0(3, 3)$ is the path algebra of $\Sigma(3, 5)$ with zero relations 
$$\beta_1\beta_2=0, \; \beta_4\beta_3=0, \; \beta_5\beta_6=0,\; \beta_8\beta_7=0$$
and the other generating relations of  $S_0(3,3)$. 

By computation we have that
$k_{\alpha}\hat{D}S_0(3, 3)k_{\alpha}$ is in fact generated by
$$y_1=\beta_2\beta_1=\beta_7\beta_8 \quad \text{ and }\quad y_2=\beta_6\beta_5=\beta_3\beta_4. $$

The vertices $1, 4, 10$ correspond to compositions $\lam=(0,0,3), \mu=(0,3,0), \nu=(3,0,0)$, respectively.
By Lemma \ref{lastlem} and similar proof or the generating relations in $\hat{D}S_0(3, 3)$, we have 
 $$k_\lam=0, \quad k_\mu=0, \quad k_\nu=0$$
and 
$$y_1y_2y_1=y_2y_1y_2=0.$$
So the algebra $k_{\alpha}\hat{D}S_0(3,3)k_{\alpha}$ is $5$-dimensional and  $1, y_1, y_2, y_1y_2, y_2y_1$ 
form a basis.
 So
$$k_{\alpha}\hat{D}S_0(3, 3)k_{\alpha}\cong NTL(3).$$
\end{example}
 
It would be interesting to find out whether $k_{\alpha}\hat{D}S_0(r, r)k_{\alpha}\cong NTL(r)$  for any $r$.

\vspace{15pt}

{\parindent=0cm
BTJ:
Department of Mathematical Sciences,
NTNU in Gj\"ovik, Norwegian University of Science and Technology,
2802 Gj\"ovik, Norway. \\
Email: bernt.jensen@ntnu.no \\}

{\parindent=0cm
XS:
Department of Mathematical Sciences,
University of Bath,
Bath BA2 7AY,
United Kingdom.\\
Email: xs214@bath.ac.uk \\}

{\parindent=0cm
GY:
School of Science,
Shandong University of Technology,
Zibo 255000 , China\\
Email: yanggy@mail.bnu.edu.cn
\end{document}

\end{document}
{\parindent=0cm
\begin{tabular}{lllll}
Bernt Tore Jensen & \hspace {2cm}& \hspace{2cm} \mbox{ }&Xiuping Su\\
NTNU in Gj\o vik && & Department of Mathematical Sciences \\
Department of Mathematical Sciences, &&& University of Bath \\
2802 Gj\o vik  &&&Bath BA2 7JY\\
Norway && &United Kingdom\\
Email: bernt.jensen@hig.no& && Email: xs214@bath.ac.uk\\\\
\end{tabular}}

\section{An example}

We end this note with an example. We describe the relations for $S_0(3, 2)$ and $DS_0(3, 2)$. Note that in this case,
all the idempotents $k_\lam$ are boundary, so $S_0(3, 2)=I(3, 2)$.
$$\xymatrix@!=0.5cm @M=0pt{
 &&&\Sigma(3,2):&\ar@/^0.2pc/[r]1
 \ar@/^0.2pc/[r]^{\beta_1}&2\ar@/^0.2pc/[r]^{\beta_2}\ar@/^0.2pc/[l]\ar@/^0.2pc/[d]^{\beta_3}
 &3\ar@/^0.2pc/[l]\ar@/^0.2pc/[d]^{\beta_4}&&&&&\\
 &&&&&\ar@/^0.2pc/[u]4\ar@/^0.2pc/[r]^{\beta_5}
&5\ar@/^0.2pc/[l]\ar@/^0.2pc/[u]\ar@/^0.2pc/[d]^{\beta_6}&&&&\\
 &&&&&&6\ar@/^0.2pc/[u]&&&&&\\
  }$$

Denote by $\frak l_i$ for $1\leq
i\leq 6$ the idempotents corresponding to the compositions
$(0,0,2), (0,1,1), (0,2,0), (1,0,1), (1,1,0), (2,0,0)$,
respectively. Let ${\beta}_i^{\rm op}$ be the opposite arrow of
$\beta_i$ for $1\leq i\leq 6$. Then $S_0(3,2)$ is generated by
$k_i, \beta_i, {\beta}^{\rm op}_i$ with nonhomogeneous relations
$${\beta}^{\rm op}_1\beta_1=k_1,~~~~ {\beta}^{\rm op}_3\beta_3=k_2,~~~~
\beta_2{\beta}^{\rm op}_2={\beta}^{\rm op}_4\beta_4=k_3,~~~$$$$
\beta_3{\beta}^{\rm op}_3={\beta}^{\rm op}_5\beta_5=k_4,~~~
\beta_5{\beta}^{\rm op}_5=k_5, ~~~\beta_6{\beta}^{\rm op}_6=k_6,$$
and homogeneous relations
$$\begin{aligned}\label{relation}
&\beta_4\beta_2\beta_1=\beta_5\beta_3\beta_1,~~~\beta_6\beta_4\beta_2=\beta_6\beta_5\beta_3,~~~
{\beta}^{\rm op}_1{\beta}^{\rm op}_2{\beta}^{\rm op}_4={\beta}^{\rm op}_1{\beta}^{\rm op}_3{\beta}^{\rm op}_5,\\~~~
&{\beta}^{\rm op}_2{\beta}^{\rm op}_4{\beta}^{\rm op}_6={\beta}^{\rm op}_3{\beta}^{\rm op}
_5{\beta}^{\rm op}_6,
\beta_1{\beta}^{\rm op}_1={\beta}^{\rm op}_2\beta_2,~~~\beta_4{\beta}^{\rm op}_4={\beta}^{\rm op}_6\beta_6.
\end{aligned}$$
The degeneration $DS_0(3,2)$ is generated by $k_i, \beta_i, {\beta}^{\rm op}_i$ with
the following zero relations
$${\beta}^{\rm op}_1\beta_1=0,~~~~ {\beta}^{\rm op}_3\beta_3=0,~~~~
\beta_2{\beta}^{\rm op}_2={\beta}^{\rm op}_4\beta_4=0,~~~$$$$
\beta_3{\beta}^{\rm op}_3={\beta}^{\rm op}_5\beta_5=0,~~~
\beta_5{\beta}^{\rm op}_5=0, ~~~\beta_6{\beta}^{\rm op}_6=0,$$
and the above homogeneous relations.

Recall that $\Sigma(3,5)$ is as follows.
$$\xymatrix@!=0.2cm @M=0pt{
 &&&&a\ar@/^0.2pc/[r]& \ar@/^0.2pc/[l]\ar@/^0.2pc/[r]\ar@/^0.2pc/[d]
 \ar@/^0.2pc/[d]&
 \ar@/^0.2pc/[r]\ar@/^0.2pc/[d]\ar@/^0.2pc/[l]&\ar@/^0.2pc/[l]\ar@/^0.2pc/[d]\ar@/^0.2pc/[r]
 &\ar@/^0.2pc/[l]\ar@/^0.2pc/[l]\ar@/^0.2pc/[d]\ar@/^0.2pc/[r]& b\ar@/^0.2pc/[l]\ar@/^0.2pc/[d]&&&&&&&\\
 &&&&&\ar@/^0.2pc/[u]\ar@/^0.2pc/[r]&\ar@/^0.2pc/[u]\ar@/^0.2pc/[l]  \ar@/^0.2pc/[u]\ar@/^0.2pc/[d]
 \ar@/^0.2pc/[r]&\ar@/^0.2pc/[l]\ar@/^0.2pc/[u]\ar@/^0.2pc/[r]\ar@/^0.2pc/[d]\ar@/^0.2pc/[u]
 &\ar@/^0.2pc/[l]\ar@/^0.2pc/[r]\ar@/^0.2pc/[u]\ar@/^0.2pc/[d]& \ar@/^0.2pc/[d]\ar@/^0.2pc/[u]\ar@/^0.2pc/[l]&&&&\\
 && &&&&\ar@/^0.2pc/[u]\ar@/^0.2pc/[r]&\ar@/^0.2pc/[u]\ar@/^0.2pc/[l]\ar@/^0.2pc/[r]
 \ar@/^0.2pc/[l]\ar@/^0.2pc/[d]& \ar@/^0.2pc/[l]\ar@/^0.2pc/[r]\ar@/^0.2pc/[u]\ar@/^0.2pc/[d]&\ar@/^0.2pc/[l]\ar@/^0.2pc/[u]\ar@/^0.2pc/[d]&&\\
 &&&&&&&
\ar@/^0.2pc/[u]\ar@/^0.2pc/[r]&\ar@/^0.2pc/[u]\ar@/^0.2pc/[l]\ar@/^0.2pc/[r]\ar@/^0.2pc/[d]&\ar@/^0.2pc/[d]\ar@/^0.2pc/[l]\ar@/^0.2pc/[u]&&&\\
 &&&&&&&&
\ar@/^0.2pc/[u]\ar@/^0.2pc/[r]&\ar@/^0.2pc/[u]\ar@/^0.2pc/[l]\ar@/^0.2pc/[d]\\
 &&&&&&&&&\ar@/^0.2pc/[u]c&&&&&\\
  }$$
where $a=k_{(0, 0, 5)}$, $b=k_{(0,  5, 0)}$ $c=k_{(5, 0, 0)}$, and
the arrows from left to right, right to left, up to down and down to up
are $e_2$, $f_2$, $e_1$ and $f_1$, respectively. Then the embedding
of $\Sigma(3,2) $ into the interior of $\Sigma(3, 5)$ in the proof of Theorem \ref{maintheorem} can be shown as follows.

$$\xymatrix@!=0.5cm @M=0pt{
&&&&a\cdot\ar@{..}[r]&\cdot\ar@{..}[r]&\cdot\ar@{..}[r]&\cdot\ar@{..}[r]&\cdot\ar@{..}[r]&\cdot b\ar@{..}[d]&& &&&&\\
 &&&& &\ar@{..}[u]\ar@{..}[r]&\ar@/^0.2pc/[r]1\ar@{..}[u]\ar@{..}[d]
 \ar@/^0.2pc/[r]^{\beta_1}&2\ar@{..}[u]\ar@/^0.2pc/[r]^{\beta_2}\ar@/^0.2pc/[l]\ar@/^0.2pc/[d]^{\beta_3}
 &3\ar@{..}[u]\ar@/^0.2pc/[l]\ar@/^0.2pc/[d]^{\beta_4}\ar@{..}[r]&\cdot\ar@{..}[d]&&&&\\
 &&
 &&&&\ar@{..}[r]&\ar@/^0.2pc/[u]4\ar@{..}[d]\ar@/^0.2pc/[r]^{\beta_5}
&5\ar@{..}[r]\ar@/^0.2pc/[l]\ar@/^0.2pc/[u]\ar@/^0.2pc/[d]^{\beta_6}&\cdot\ar@{..}[d]&&&\\
 &&&&&&&\ar@{..}[r]&6\ar@{..}[r]\ar@{..}[d]\ar@/^0.2pc/[u]&\cdot\ar@{..}[d]&&&&\\
 &&&&&&&&\ar@{..}[r]&\cdot\ar@{..}[d]&\\
 &&&&&&&&& c&&\\
  }$$